\documentclass{birkjour}
\usepackage{hyperref,color}

 \newtheorem{theorem}{Theorem}[section]
 \newtheorem{corollary}[theorem]{Corollary}
 \newtheorem{lemma}[theorem]{Lemma}
 \numberwithin{equation}{section}
\newcommand{\cB}{\mathcal{B}}
\newcommand{\cP}{\mathcal{P}}
\newcommand{\C}{\mathbb{C}}
\newcommand{\N}{\mathbb{N}}
\newcommand{\T}{\mathbb{T}}
\newcommand{\R}{\mathbb{R}}
\newcommand{\Z}{\mathbb{Z}}
\begin{document}
\title[{More on the Density of Analytic Polynomials}]
{More on the Density of Analytic\\ 
Polynomials in Abstract Hardy Spaces}
\author{Alexei Karlovich}
\address{%
Centro de Matem\'atica e Aplica\c{c}\~oes,\\
Departamento de Matem\'a\-tica, \\
Faculdade de Ci\^encias e Tecnologia,\\
Universidade Nova de Lisboa,\\
Quinta da Torre, \\
2829--516 Caparica, Portugal}
\email{oyk@fct.unl.pt}
\author{Eugene Shargorodsky}
\address{%
Department of Mathematics\\
King's College London\\
Strand, London WC2R 2LS\\
United Kingdom}
\email{eugene.shargorodsky@kcl.ac.uk}
\thanks{%
This work was partially supported by the Funda\c{c}\~ao para a Ci\^encia e a
Tecnologia (Portu\-guese Foundation for Science and Technology)
through the project
UID/MAT/00297/2013 (Centro de Matem\'atica e Aplica\c{c}\~oes).}
\begin{abstract}
Let $\{F_n\}$ be the sequence of the Fej\'er kernels on the unit circle $\T$.
The first author {recently proved} that if $X$ is a separable Banach 
function space on $\T$ such that the Hardy-Littlewood maximal
operator $M$ is bounded on its associate space $X'$, then $\|f*F_n-f\|_X\to 0$
for every $f\in X$ as $n\to\infty$. This implies that the set of analytic 
polynomials $\cP_A$ is dense in the abstract Hardy space $H[X]$ built upon 
a separable Banach function space $X$ such that $M$ is bounded on $X'$.
In this note we show that there exists a separable weighted $L^1$ space $X$ 
such that the sequence $f*F_n$ does not 
always converge to $f\in X$ in the norm of $X$. On the other hand,
we prove that the set $\cP_A$ is dense in $H[X]$ under the assumption that 
$X$ is merely separable. 
\end{abstract}


\maketitle
\section{Preliminaries and the main results}
For $0< p\le\infty$, let $L^p:=L^p(\T)$ be the Lebesgue space on the unit
circle $\T:=\{z\in\C:|z|=1\}$ in the complex plane $\C$. For $f\in L^1$, let
\[
\widehat{f}(n):=\frac{1}{2\pi}
\int_{-\pi}^\pi f(e^{i\theta})e^{-in\theta}\,d\theta,
\quad n\in\Z,
\]
be the sequence of the Fourier coefficients of $f$.
Let $X$ be a Banach space continuously embedded in $L^1$.
Following \cite[p.~877]{Xu92}, we will consider the abstract Hardy space 
$H[X]$ built upon the space $X$, which is defined by
\[
H[X]:=\big\{f\in X:\ \widehat{f}(n)=0\quad\mbox{for all}\quad n<0\big\}.
\]
It is clear that if $1\le p\le\infty$, then $H[L^p]$ is the classical Hardy
space $H^p$.

A function of the form
\[
q(t)=\sum_{k=0}^n\alpha_k t^k,
\quad
t\in\T,
\quad
\alpha_0,\dots,\alpha_n\in\C,
\]
is said to be an analytic polynomial on $\T$. The set of all analytic
polynomials is denoted by $\cP_A$. It is well known {that} the set
$\cP_A$ is dense in $H^p$ whenever $1\le p<\infty$
(see, e.g., \cite[Chap.~III, Corollary~1.7(a)]{C91}).
The density of the set $\cP_A$ in the abstract Hardy spaces $H[X]$
was studied by the first author \cite{K-CM} for the case when
$X$ is a so-called Banach function space. 

Let us recall the definition of a Banach function space.
We equip $\T$ with the normalized Lebesgue measure $dm(t)=|dt|/(2\pi)$.
Let $L^0$ be the space of all measurable complex-valued functions on $\T$. 
As usual, we do not distinguish {functions which} are equal almost everywhere
(for the latter we use the standard abbreviation a.e.).  Let $L^0_+$ be the
subset of functions in $L^0$ whose values lie in $[0,\infty]$. The
characteristic function of a measurable set $E\subset\T$ is denoted by 
$\chi_E$.

Following \cite[Chap.~1, Definition~1.1]{BS88}, a mapping 
$\rho: L_+^0\to [0,\infty]$ is called a Banach function norm
if, for all functions $f,g, f_n\in L_+^0$ with $n\in\N$, for all
constants $a\ge 0$, and for all measurable subsets $E$ of $\T$, the
following  properties hold:
\begin{eqnarray*}
{\rm (A1)} & &
\rho(f)=0  \Leftrightarrow  f=0\ \mbox{a.e.},
\
\rho(af)=a\rho(f),
\
\rho(f+g) \le \rho(f)+\rho(g),\\
{\rm (A2)} & &0\le g \le f \ \mbox{a.e.} \ \Rightarrow \ 
\rho(g) \le \rho(f)
\quad\mbox{(the lattice property)},\\
{\rm (A3)} & &0\le f_n \uparrow f \ \mbox{a.e.} \ \Rightarrow \
       \rho(f_n) \uparrow \rho(f)\quad\mbox{(the Fatou property)},\\
{\rm (A4)} & & m(E)<\infty\ \Rightarrow\ \rho(\chi_E) <\infty,\\
{\rm (A5)} & &\int_E f(t)\,dm(t) \le C_E\rho(f)
\end{eqnarray*}
with {a constant} $C_E \in (0,\infty)$ that may depend on $E$ and 
$\rho$,  but is independent of $f$. When functions differing only on 
a set of measure  zero are identified, the set $X$ of all functions 
$f\in L^0$ for  which  $\rho(|f|)<\infty$ is called a Banach function
space. For each $f\in X$, the norm of $f$ is defined by
$\|f\|_X :=\rho(|f|)$.
The set $X$ under the natural linear space operations and under 
this norm becomes a Banach space (see 
\cite[Chap.~1, Theorems~1.4 and~1.6]{BS88}). 
If $\rho$ is a Banach function norm, its associate norm 
$\rho'$ is defined on $L_+^0$ by
\[
\rho'(g):=\sup\left\{
\int_\T f(t)g(t)\,dm(t) \ : \ 
f\in L_+^0, \ \rho(f) \le 1
\right\}, \ g\in L_+^0.
\]
It is a Banach function norm itself \cite[Chap.~1, Theorem~2.2]{BS88}.
The Banach function space $X'$ determined by the Banach function norm
$\rho'$ is called the associate space (K\"othe dual) of $X$. 
The associate space $X'$ can be viewed as a subspace of the (Banach) 
dual space $X^*$. 

Recall that $L^1$ is a commutative Banach algebra under the convolution
multiplication defined for $f,g\in L^1$ by
\[
(f* g)(e^{i\theta})=\frac{1}{2\pi}\int_{-\pi}^\pi 
f(e^{i\theta-i\varphi})g(e^{i\varphi})\,d\varphi,
\quad
e^{i\theta}\in\T.
\]
For $n\in\N$, let
\[
F_n(e^{i\theta}):=\sum_{k=-n}^n\left(1-\frac{|k|}{n+1}\right)e^{i\theta k}
=
\frac{1}{n+1}\left(
\frac{\sin\frac{n+1}{2}\theta}{\sin\frac{\theta}{2}}
\right)^2,
\quad e^{i\theta}\in\T,
\]
be the $n$-th Fej\'er kernel. For $f\in L^1$, the $n$-th Fej\'er mean of $f$ 
is defined as the convolution $f*F_n$. 

Given $f\in L^1$, the Hardy-Littlewood maximal function is defined by
\[
(Mf)(t):=\sup_{I\ni t}\frac{1}{m(I)}\int_I|f(\tau)|\,dm(\tau),
\quad t\in\T,
\]
where the supremum is taken over all arcs  $I\subset\T$ containing $t\in\T$.
The operator $f\mapsto Mf$ is called the Hardy-Littlewood maximal operator.
\begin{theorem}[{\cite[Theorem~3.3]{K-CM}}]
\label{th:Fejer-norm-convergence}
Suppose $X$ is a separable Banach function space on $\T$. If the 
Hardy-Littlewood maximal operator is bounded on the associate space $X'$, 
then for every $f\in X$,
\begin{equation}\label{eq:Fejer-norm-convergence}
\lim_{n\to\infty}\|f*F_n-f\|_X=0.
\end{equation}
\end{theorem}
It is well known that for $f\in L^1$ one has
\[
(f*F_n)(e^{i\theta})
=
\sum_{k=-n}^n
\widehat{f}(k)\left(1-\frac{|k|}{n+1}\right)e^{i\theta k},
\quad e^{i\theta}\in\T
\]
(see, e.g., \cite[Chap.~I]{Kat76}). This implies that if 
$f\in H[X]\subset H[L^1]=H^1$, then $f*F_n\in\cP_A$. Combining this observation
with Theorem~\ref{th:Fejer-norm-convergence}, we arrive at the following.
\begin{corollary}[{\cite[Theorem~1.2]{K-CM}}]
\label{co:density}
Suppose $X$ is a separable Banach function space on $\T$. If the 
Hardy-Littlewood maximal operator $M$ is bounded on its associate space 
$X'$, then the set of analytic polynomials $\cP_A$ is dense in the 
abstract Hardy space $H[X]$ built upon the space $X$.
\end{corollary} 
Note that if a Banach function space $X$ is, in addition, 
rearrangement-invariant then the requirement of the boundedness of $M$
on the space $X'$ can be omitted in Corollary~\ref{co:density}
(see \cite[Theorem~1.1]{K-CM} or \cite[Lemma~1.3(c)]{L-arXiv}). 
Le\'snik \cite{L-conjecture} conjectured that the same fact should be true 
for arbitrary, not necessarily rearrangement-invariant, Banach function 
spaces.

In this note, we first observe that Theorem~\ref{th:Fejer-norm-convergence}
does not hold for arbitrary separable Banach function spaces.
For a function $K\in L^1$, consider the convolution operator 
$C_K$ with kernel $K$ defined by
\[
C_K f=f*K, \quad f\in L^1.
\]

It follows from \cite[Theorem~2]{Sh96} that there exists a continuous
function $p:\T\to[1,\infty)$ such that the sequence of the convolution
operators $C_{F_n}$ is not uniformly bounded in the variable Lebesgue
space $L^{p(\cdot)}$ defined as the set of all $f\in L^0$ such that
\[
\int_\T |f(t)|^{p(t)}dm(t)<\infty.
\]
It is well known (see, e.g., \cite[Propostion~2.12, Theorem~2.78, 
Section~2.10.3]{CF13}) that if $p:\T\to[1,\infty)$ is continuous, then
$L^{p(\cdot)}$ is a separable Banach function 
space equipped with the norm
\[
\|f\|_{L^{p(\cdot)}}=\inf\left\{\lambda>0\ :\
\int_\T \left|\frac{f(t)}{\lambda}\right|^{p(t)}dm(t) \le 1
\right\}.
\]
Since the norms of the convolution operators $C_{F_n}$ may not be uniformly 
bounded on $L^{p(\cdot)}$, the standard argument, based on the uniform
boundedness principle, leads us to the following.
\begin{theorem}
\label{th:from-Sharapudinov}
There exist a separable Banach function space $X$ on $\T$ and a function 
$f\in X$ such that \eqref{eq:Fejer-norm-convergence} is not fulfilled.
\end{theorem}
We show that the separable Banach function space in 
Theorem~\ref{th:from-Sharapudinov} can be chosen as a weighted $L^1$ space,
that is, the techniques of variable Lebesgue spaces can be
omitted. 
\begin{theorem}[Main result 1]
\label{th:main-1}
There exist a nonnegative function $w\in L^1$ such that $w^{-1}\in L^\infty$
and a function $f$ in the separable Banach function space 
\[
X=L^1(w)=\{f\in L^0: fw\in L^1\}
\]
such that \eqref{eq:Fejer-norm-convergence} is not fulfilled.
\end{theorem}
In spite of the observation made in Theorems~\ref{th:from-Sharapudinov}
and~\ref{th:main-1}, we show that the requirement of the boundedness
of the Hardy-Littlewood maximal operator $M$ on the associate space 
$X'$ of a separable Banach function space $X$ in Corollary~\ref{co:density}
can be omitted. Thus, Le\'snik's conjecture \cite{L-conjecture} is, indeed, 
true.
\begin{theorem}[Main result 2]
\label{th:main-2}
If $X$ is a separable Banach function space on $\T$, then the set of analytic
polynomials $\cP_A$ is dense in the abstract Hardy space $H[X]$
built upon the space $X$.
\end{theorem}
The paper is organized as follows.
In Section~\ref{sec:first-proof}, we prove that a convolution operator
$C_K$ with a nonnegative symmetric kernel $K\in L^1$ is bounded
on a Banach function space $X$ if and only if it is bounded in its associate
space $X'$. Further, we consider a special weight $w\in L^1$ such that 
$w^{-1}\in L^\infty$. Then $X=L^1(w)$ is a separable Banach function space 
with the associate space $X'=L^\infty(w^{-1})$. We show that the sequence of 
convolution operators $\{C_{K_n}\}$ with nonnegative bounded symmetric 
kernels $K_n$, satisfying $\|K_n\|_{L^1}=1$ and a natural localization
property, is not uniformly bounded on $X'=L^\infty(w^{-1})$, and therefore,
on its associate space $X''=X=L^1(w)$. Applying this result to the sequence of the 
Fej\'er kernels $\{F_n\}$, we prove Theorem~\ref{th:main-1} with the aid of 
the uniform boundedness principle.

In Section~\ref{sec:second-proof}, we recall that the separability of a Banach
function space $X$ is equivalent to $X^*=X'$. Further, we collect some
facts on the identification of the Hardy spaces $H^p$ on the unit circle
and the Hardy spaces $H^p(\mathbb{D})$ of analytic functions in the unit disk
$\mathbb{D}$. Finally, we give the proof of Theorem~\ref{th:main-2}
based on the application of the Hahn-Banach theorem, a corollary of the 
Smirnov theorem and properties of the identification of $H^1$ with 
$H^1(\mathbb{D})$.
\section{Proof of the first main result}\label{sec:first-proof}
\subsection{Norms of convolution operators on $X$ and on its associate space $X'$}
The Banach space of all bounded linear operator on a Banach space $E$ is 
denoted by $\mathcal{B}(E)$. 
\begin{lemma}\label{le:convolution-BFS-associate}
Let $X$ be a Banach function space on $\T$ and $K\in L^1$ be a nonnegative 
function such that $K(e^{i\theta})=K(e^{-i\theta})$ for almost all 
$\theta\in[-\pi,\pi]$.  Then the convolution operator $C_K$ is bounded on the 
Banach function $X$ if and only if it is bounded on its associate space 
{$X'$. In that case}
\begin{equation}\label{eq:convolution-BFS-associate-1}
\|C_K\|_{\mathcal{B}(X')}
=
\|C_K\|_{\mathcal{B}(X)} .
\end{equation}
\end{lemma}
\begin{proof}
Suppose $C_K$ is bounded on $X'$. Fix $f\in X\setminus\{0\}$. Since $K\ge 0$, 
we have  $|f*K|\le|f|*K$. According to the Lorentz-Luxemburg theorem 
(see, e.g., \cite[Chap.~1, Theorem~2.7]{BS88}), $X=X''$ with equality of the norms. Hence
\begin{align*}
\|f*K\|_X
&\le 
\|\,|f|*K\|_X=\|\,|f|*K\|_{X''}
\\
&=
\sup\left\{
\int_\T(|f|*K)(t)|g(t)|\,dm(t)\ : \ g\in X',\ \|g\|_{X'}\le 1
\right\}.
\end{align*}
Then for every $\varepsilon > 0$ there exists a function $h\in X'$ such that $h\ge 0$, $\|h\|_{X'}\le 1$,
and
\begin{equation}\label{eq:convolution-BFS-associate-2}
\|f*K\|_X\le (1 + \varepsilon)\int_\T(|f|*K)(t)h(t)\,dm(t).
\end{equation}
Taking into account that $K(e^{i\theta})=K(e^{-i\theta})$
for almost all $\theta\in\R$, by Fubini's theorem, we get
\[
\int_\T(|f|*K)(t)h(t)\,dm(t)=\int_\T(h*K)(t)|f(t)|\,dm(t).
\]
From this identity, H\"older's inequality for $X$
(see, e.g., \cite[Chap.~1, Theorem~2.4]{BS88}), 
and the boundedness of $C_K$ on $X'$,
we obtain
\begin{equation}\label{eq:convolution-BFS-associate-3}
\int_\T(|f|*K)(t)h(t)\,dm(t)
\le
\|f\|_X\|h*K\|_{X'}
\le 
\|f\|_X\|C_K\|_{\cB(X')}.
\end{equation}
It follows from \eqref{eq:convolution-BFS-associate-2}--\eqref{eq:convolution-BFS-associate-3}
that
$$
\|C_K\|_{\cB(X)}=\sup_{f\in X,f\ne 0}\frac{\|f*K\|_X}{\|f\|_X}\le (1 + \varepsilon)\|C_K\|_{\cB(X')}
$$
for every $\varepsilon > 0$,
which implies the boundedness of $C_K$ on $X$ and the inequality 
\begin{equation}\label{eq:convolution-BFS-associate-4}
\|C_K\|_{\cB(X)} \le \|C_K\|_{\cB(X')}.
\end{equation}

If $C_K$ is bounded on $X$, then using
the Lorentz-{Luxemburg} theorem and \eqref{eq:convolution-BFS-associate-4}
with $X'$ in place of $X$, we obtain that $C_K$ is bounded on $X'$ and
\begin{equation}\label{eq:convolution-BFS-associate-5}
\|C_K\|_{\cB(X')}\le \|C_K\|_{\cB(X'')}=\|C_K\|_{\cB(X)}.
\end{equation}
Combining \eqref{eq:convolution-BFS-associate-4}--\eqref{eq:convolution-BFS-associate-5},
we arrive at \eqref{eq:convolution-BFS-associate-1}.
\end{proof}
\subsection{Spaces $L^1(w)$ and $L^\infty(w^{-1})$ with a special weight $w$}
\begin{lemma}\label{le:weighted-L1-L-infty}
Let
\begin{equation}\label{eq:weight}
w\left(e^{i\theta}\right) := \left\{\begin{array}{cll}
\sqrt{m},  & 
\frac{\pi}{2m} \le |\theta| \le \frac{\pi}{2m - 1}, &
m \in \mathbb{N}  ,  
\\[3mm]
1, &   \frac{\pi}{2m + 1} < |\theta| < \frac{\pi}{2m},
&
m \in \mathbb{N}.
\end{array}\right.
\end{equation}
Then the spaces 
\[
L^1(w)=\{f\in L^0:fw\in L^1\},
\quad
L^\infty(w^{-1})=\{f\in L^0:fw^{-1}\in L^\infty\}
\]
are Banach function spaces on $\T$ with respect to the norms
\[
\|f\|_{L^1(w)}=\|fw\|_{L^1},
\quad
\|f\|_{L^\infty(w^{-1})}=\|fw^{-1}\|_{L^\infty},
\]
and $(L^1(w))'=L^\infty(w^{-1})$. Moreover, the space $L^1(w)$ is separable.
\end{lemma}
\begin{proof}
It is clear that $w^{-1}\in L^\infty$ and, since
\begin{align}
\|w\|_{L^1}
&=
\frac{1}{2\pi}\int_{-\pi}^\pi w(e^{i\theta})\,d\theta
\nonumber\\
&=
\sum_{m=1}^\infty\left(\frac{1}{2m}-\frac{1}{2m+1}\right)
+
\sum_{m=1}^\infty\sqrt{m}\left(\frac{1}{2m-1}-\frac{1}{2m}\right)<\infty,
\label{eq:weight-L1}
\end{align}
we also have $w\in L^1$. Then it follows from \cite[Lemma~2.5]{K03} that
$L^1(w)$ and $L^\infty(w^{-1})$ are Banach function spaces and
$(L^1(w))'=L^\infty(w^{-1})$. Finally, the separability of the space $L^1(w)$ 
follows from \cite[Proposition~2.6]{K03} and 
\cite[Chap.~1, Corollary~5.6]{BS88}.
\end{proof}
\subsection{Norms of convolution operators are not uniformly bounded
on {the} spaces $L^1(w)$ and $L^\infty(w^{-1})$ with 
{the} special weight $w$}
\begin{theorem}\label{th:non-uniform-bounedeness}
Let $\{K_n\}$ be a sequence of bounded functions $K_n:\T\to\C$
such that
\begin{equation}\label{eq:kernel-1}
K_n(e^{i\theta})\ge 0,\quad K_n(e^{i\theta})=K_n(e^{-i\theta})
\quad\mbox{a.e. on}\quad[-\pi,\pi],
\end{equation}
\begin{equation}\label{eq:kernel-2}
\frac{1}{2\pi}\int_{-\pi}^\pi K_n\left(e^{i\theta}\right)\, d\theta = 1,
\end{equation}
and
\begin{equation}\label{eq:kernel-3}
\lim_{n \to \infty}\, \sup_{\varepsilon \le |\theta| \le \pi} 
K_n\left(e^{i\theta}\right) = 0  \quad\mbox{for {each}}
\quad \varepsilon > 0.
\end{equation}
If $w$ is the weight given by \eqref{eq:weight}, then the convolution
operators $C_{K_n}$ are bounded on $L^\infty(w^{-1})$ and on $L^1(w)$
for all $n\in\N$, however,
\begin{align}
&\sup_{n\in\N}\|C_{K_n}\|_{\mathcal{B}(L^\infty(w^{-1}))}=\infty,
\label{eq:non-uniform-boundedness-1}
\\
&\sup_{n\in\N}\|C_{K_n}\|_{\mathcal{B}(L^1(w))}=\infty.
\label{eq:non-uniform-boundedness-2}
\end{align}
\end{theorem}
\begin{proof}
By \eqref{eq:weight}--\eqref{eq:weight-L1}, $w\in L^1$ and 
$w^{-1}\in L^\infty$. Therefore, for every $n\in\N$,
\begin{align*}
\|C_{K_n} f\|_{L^1(w)} 
&\le  
\frac{1}{2\pi}\left\|
\int_{-\pi}^\pi 
K_n(e^{i(\cdot - \theta)}) 
\left|f\left(e^{i\theta}\right)\right|\, d\theta
\right\|_{L^1(w)} 
\\
& \le 
\frac{1}{2\pi}
\int_{-\pi}^\pi \left\|K_n(e^{i(\cdot - \theta)})\right\|_{L^1(w)} 
\left|f\left(e^{i\theta}\right)\right|\, d\theta 
\\
& 
\le 
\frac{1}{2\pi} \|K_n\|_{L^\infty} \|w\|_{L^1} \|f\|_{L^1} 
\\
&=
\frac{1}{2\pi} \|K_n\|_{L^\infty} \|w\|_{L^1} \|w^{-1}fw\|_{L^1}
\\
&\le 
\frac{1}{2\pi} 
\|K_n\|_{L^\infty} \|w\|_{L^1} \|w^{-1}\|_{L^\infty}\|f\|_{L^1(w)}.
\end{align*}
Hence
\[
\|C_{K_n}\|_{\cB(L^1(w))} 
\le 
\frac{1}{2\pi} 
\|K_n\|_{L^\infty} \|w\|_{L^1} \|w^{-1}\|_{L^\infty},
\quad
n\in\N.
\]
It follows from \eqref{eq:kernel-1} and 
Lemmas~\ref{le:convolution-BFS-associate}--\ref{le:weighted-L1-L-infty} that 
the operators $C_{K_n}$ are bounded on $L^\infty(w^{-1})$ for all $n\in\N$. 
Moreover, \eqref{eq:non-uniform-boundedness-1} implies 
\eqref{eq:non-uniform-boundedness-2}.

Let us prove \eqref{eq:non-uniform-boundedness-1}.
Consider the sequence
\[
v_m\left(e^{i\theta}\right) 
:= 
\left\{\begin{array}{cl}
\sqrt{m},  & \frac{\pi}{2m} \le \theta \le \frac{\pi}{2m - 1}\,, 
\\[3mm]
0 , &   
\theta \in [-\pi, \pi]\setminus \left[\frac{\pi}{2m}, \frac{\pi}{2m - 1}\right],
\end{array}\right.
\quad
m\in\N.
\]
Then it follows from \eqref{eq:weight} that
$\|v_m\|_{L^\infty(w^{-1})} = 1$ for all $m\in\N$.

Fix $m\in\N$.
According to \eqref{eq:kernel-2} and the localization property
\eqref{eq:kernel-3}, there exists $n(m)\in\N$ such that
\[
\int_{-\frac{\pi}{(2m)^2}}^0 K_n\left(e^{i\theta}\right)\, d\theta 
= 
\frac{1}{2} 
\int_{-\frac{\pi}{(2m)^2}}^{\frac{\pi}{(2m)^2}} K_n\left(e^{i\theta}\right)\, d\theta 
\ge 
\frac{1}{3} 
\quad \mbox{for all}\quad n \ge n(m). 
\]
Since $K_n\in L^1$, for every $n \ge n(m)$, there exists $\delta_n > 0$ such 
that
\[
\int_{-\frac{\pi}{(2m)^2}}^{-\delta_n} 
K_n\left(e^{i\theta}\right)\, d\theta \ge \frac{1}{4}.
\]
Therefore, for almost all
$\vartheta \in \left[\frac{\pi}{2m}-\delta_n,\frac{\pi}{2m}\right]$, one gets
\begin{align}
\left(C_{K_n} v_m\right)\left(e^{i\vartheta}\right) 
&= 
\frac{\sqrt{m}}{2\pi} 
\int_{ \frac{\pi}{2m}}^{\frac{\pi}{2m - 1}} 
K_n\left(e^{i\vartheta - i\theta}\right)\, d\theta
\notag\\
&\ge 
\frac{\sqrt{m}}{2\pi} 
\int_{ \frac{\pi}{2m}}^{\frac{\pi}{2m} + \frac{\pi}{(2m)^2}} 
K_n\left(e^{i\vartheta - i\theta}\right)\, d\theta 
\notag\\
&=
\frac{\sqrt{m}}{2\pi}
\int_{\vartheta-\frac{\pi}{2m}-\frac{\pi}{(2m)^2}}^{\vartheta-\frac{\pi}{2m}}
K_n\left(e^{i\eta}\right)\, d\eta
\notag\\
&\ge 
\frac{\sqrt{m}}{2\pi} 
\int_{-\frac{\pi}{(2m)^2}}^{-\delta_n} K_n\left(e^{i\eta}\right)\, d\eta 
\ge \frac{\sqrt{m}}{8\pi}.
\label{eq:non-uniform-boundedness-3}
\end{align}
In view of \eqref{eq:weight}, $w(e^{i\vartheta})=1$ for all
$\vartheta\in\left(
\max\left\{\frac{\pi}{2m}-\delta_n,\frac{\pi}{2m+1}\right\},\frac{\pi}{2m}
\right)$. Hence, it follows from \eqref{eq:non-uniform-boundedness-3}
that
\[
\|C_{K_n} v_m\|_{L^\infty(w^{-1})} \ge \frac{\sqrt{m}}{8\pi}
\quad\mbox{for all}\quad n \ge n(m),
\]
while $\|v_m\|_{L^\infty(w^{-1})} = 1$. So
\[
\|C_{K_n}\|_{\cB(L^\infty(w^{-1}))} \ge \frac{\sqrt{m}}{8\pi} 
\quad\mbox{for all}\quad
n \ge n(m).
\]
Since $m\in\N$ is arbitrary, the latter inequality immediately
implies \eqref{eq:non-uniform-boundedness-1}.
\end{proof}
\subsection{Proof of Theorem~\ref{th:main-1}}
Let $X=L^1(w)$, where $w$ is the weight given by \eqref{eq:weight}. By
Lemma~\ref{le:weighted-L1-L-infty}, $X$ is a separable Banach function
space. It is well known (and not difficult to check) that the sequence
$\{F_n\}$ of the Fej\'er kernels is a sequence of bounded functions satisfying
\eqref{eq:kernel-1}--\eqref{eq:kernel-3}. By 
Theorem~\ref{th:non-uniform-bounedeness}, the operators $C_{F_n}$ are bounded
on $X$ for every $n\in\N$. 

Assume that \eqref{eq:Fejer-norm-convergence}
is fulfilled for all $f\in X$. Then, for all $f\in X$, the sequence 
$\{C_{F_n}f\}$ is bounded in $X$. Therefore, by the uniform boundedness 
principle, the sequence $\{\|C_{F_n}\|_{\cB(X)}\}$ is bounded, but
this contradicts \eqref{eq:non-uniform-boundedness-2}. Thus, there
exists a function $f\in X$ such that \eqref{eq:Fejer-norm-convergence}
does not hold.
\qed
\section{Proof of the second main result}\label{sec:second-proof}
\subsection{Separable Banach function spaces $X$ are spaces for which $X^*$ 
is\\ isometrically isomorphic to $X'$}
Combining \cite[Chap.~I, Corollaries 4.3 and 5.6]{BS88} and observing that
the measure $dm$ is separable (for the definition of a separable measure, 
see, e.g., \cite[p. 27]{BS88} or \cite[Chap.~I, Section 6.10]{KA82}), we arrive at 
the following.
\begin{theorem}\label{th:dual-associate}
Let $X$ be a Banach function space on $\T$. Then $X$ is separable if and only 
if its dual space $X^*$ is isometrically isomorphic to the associate space $X'$.
\end{theorem}
\subsection{Hardy spaces on the unit disk}
Let $\mathbb{D}$ denote the open unit disk in the complex plane $\C$. Recall
that a function $F$ analytic in $\mathbb{D}$ is said to belong to the Hardy 
space $H^p(\mathbb{D})$, $0<p\le\infty$, if the integral mean
\begin{align*}
&
M_p(r,F)=\left(\frac{1}{2\pi}\int_{-\pi}^\pi |F(re^{i\theta})|^p\,d\theta\right)^{1/p},
\quad
0<p<\infty,
\\
&
M_\infty(r,F)=\max_{-\pi\le\theta\le\pi}|F(re^{i\theta})|,
\end{align*}
remains bounded as $r\to 1$. If $F\in H^p(\mathbb{D})$, $0<p\le\infty$, then
the nontangential limit
\[
f(e^{i\theta})=\lim_{r\to 1-0}F(re^{i\theta})
\]
exists for almost all $\theta\in[-\pi,\pi]$ (see, e.g., \cite[Theorem~2.2]{D70})
and the boundary function $f=f(e^{i\theta})$ belongs to $L^p$.

The following lemma is an immediate consequence of the Smirnov theorem
(see, e.g., \cite[Theorem~2.1]{D70}).
\begin{lemma}\label{le:Smirnov}
If $F\in H^p(\mathbb{D})$ for some $p\in(0,1)$ and its boundary function
$f$ belongs to $L^1$, then $F\in H^1(\mathbb{D})$.
\end{lemma}
Recall that if $f\in H^1$ then its analytic extension $F$ into 
$\mathbb{D}$, given by the Poisson integral 
\[
F(re^{i\theta})
=
\frac{1}{2\pi}\int_{-\pi}^\pi P(r,\theta-\varphi) f(e^{i\varphi})\,d\varphi,
\quad
0\le r<1,\quad -\pi\le\theta\le\pi,
\]
where
\[
P(r,\theta)
=
\frac{1-r^2}{1-2r\cos\theta+r^2},
\quad
0\le r<1,\quad -\pi\le\theta\le\pi,
\]
is the Poisson kernel, belongs to $H^1(\mathbb{D})$ and the boundary
function of $F$ coincides with $f$ a.e. on $\T$ (see, e.g., 
\cite[Theorem~3.1]{D70}).

It is important to note that the Taylor coefficients of
$F\in H^p(\mathbb{D})$ coincide with the Fourier coefficients of its
boundary function $f\in L^p$. More precisely, one has the following.
\begin{theorem}[{\cite[Theorem~3.4]{D70}}]
\label{th:H1-Taylor-Fourier}
Let $F(z)=\sum_{n=0}^\infty a_n z^n$ belong to $H^1(\mathbb{D})$ and
let $\{\widehat{f}(n)\}$ be the sequence of the Fourier coefficients of its 
boundary function $f\in L^1$. Then $\widehat{f}(n)=a_n$ for all $n\ge 0$
and $\widehat{f}(n)=0$ for $n<0$.
\end{theorem}
\subsection{Proof of Theorem~\ref{th:main-2}}
Suppose $\cP_A$ is not dense in $H[X]$. Take any function $f \in H[X]$ that 
does not belong to the closure of $\cP_A$ with respect to the norm of $X$. 
Since $X$ is separable, it follows from Theorem~\ref{th:dual-associate}
that $X^*$ is isometrically isomorphic to $X'$. Then, by a corollary of 
the Hahn-Banach theorem
(see, e.g., \cite[Chap.~7, Theorem~4.1]{BSU96}), there exists
a function $g \in X' \subset L^1$ such that
\begin{equation}\label{eq:main-1}
\int_{-\pi}^\pi f\left(e^{i\theta}\right)g\left(e^{i\theta}\right)\, d\theta \not= 0 
\end{equation}
and
\[
\int_{-\pi}^\pi p\left(e^{i\theta}\right)g\left(e^{i\theta}\right)\, d\theta = 0 
\quad\mbox{for all}\quad p \in \mathcal{P}_A.
\]
In particular, if $p(e^{i\theta})=e^{in\theta}$ with $n=0,1,2,\dots$, then
\begin{equation}\label{eq:main-2}
\widehat{g}(-n)=0
\quad\mbox{for all}\quad
n = 0, 1, 2, \dots.
\end{equation}
Hence $g\in H[X']\subset H^1$.
For functions $f\in H[X]\subset H^1$ and $g\in H[X']\subset H^1$, let
$F$ and $G$ denote their analytic extensions to the unit disk $\mathbb{D}$
by means of their Poisson integrals. Then $F,G\in H^1(\mathbb{D})$.
It follows from \eqref{eq:main-2} and Theorem~\ref{th:H1-Taylor-Fourier}
that $G(0) = 0$. Since $F,G\in H^1(\mathbb{D})$, by H\"older's inequality, 
$FG \in H^{1/2}(\mathbb{D})$. On the other hand, since $f\in X$ 
and $g\in X'$,
it follows from H\"older's inequality for Banach function spaces (see 
\cite[Chap.~1, Theorem~2.4]{BS88}) that $fg \in L^1$. Then 
it follows from Lemma~\ref{le:Smirnov} that $FG \in H^1(\mathbb{D})$. Since 
$(FG)(0) = F(0)G(0) = 0$, applying Theorem~\ref{th:H1-Taylor-Fourier}
to $FG$, we obtain $\widehat{fg}(0)=0$, that is,
\[
\int_{-\pi}^\pi f\left(e^{i\theta}\right)g\left(e^{i\theta}\right)\, d\theta = 0 ,
\]
which contradicts \eqref{eq:main-1}.
\qed
\subsection*{Acknowledgment}
We would like to thank the referee for the useful remarks. 


\begin{thebibliography}{X}
\bibitem{BS88} 
C.~Bennett and R.~Sharpley, 
\textit{Interpolation of Operators}.
Academic Press, Boston, 1988. 

\bibitem{BSU96}
Yu. M. Berezansky, Z. G. Sheftel, and G. F. Us, 
\textit{Functional Analysis, Vol. 1},
Birkh\"auser, Basel, 1996.

\bibitem{C91}
J.~B.~Conway, 
\textit{The Theory of Subnormal Operators}.
American Mathematical Society, Providence, RI, 1991.

\bibitem{CF13}
D. Cruz-Uribe and A. Fiorenza,
\textit{Variable Lebesgue Spaces}. 
Birkh\"auser, Basel, 2013.

\bibitem{D70} 
P.~L.~Duren, 
\textit{Theory of $H^p$ Spaces}.
Academic Press, New York and London, 1970.

\bibitem{KA82}
L.~V.~Kantorovich and G.~P.~Akilov,
{\em Functional Analysis}.
Pergamon Press, Oxford, 2nd ed., 1982.

\bibitem{K03}
A.~Karlovich,
\textit{Fredholmness of singular integral operators with piecewise 
continuous coefficients on weighted Banach function spaces}. 
J. Integral Equations Appl. \textbf{15} (2003), 263--320.

\bibitem{K-CM}
A.~Karlovich,
\textit{Density of analytic polynomials in abstract Hardy spaces}.
{Comment. Math.}, to appear. Preprint is available
at arXiv:1710.10078 [math.CA] (2017).

\bibitem{Kat76} 
Y. Katznelson, 
\textit{An Introduction to Harmonic Analysis}, 
Dower Publications, Inc., New York, 1976.

\bibitem{L-conjecture}
K.~Le\'snik,
\textit{Personal communication to A. Karlovich}. February 23, 2017.

\bibitem{L-arXiv}
K.~Le\'snik,
\textit{Toeplitz and Hankel operators between distinct Hardy spaces}.
arXiv:1708.00910 [math.FA] (2017).

\bibitem{Sh96}
I. I. Sharapudinov,
\textit{Uniform boundedness in $L^p$ ($p=p(x)$) of some families of 
convolution operators}.
Math. Notes \textbf{59} (1996), 205--212.

\bibitem{Xu92}
Q.~Xu, \textit{Notes on interpolation of Hardy spaces}.
Ann. Inst. Fourier \textbf{42} (1992), 875--889.
\end{thebibliography}
\end{document}